\documentclass[11pt]{amsart}
\usepackage{amssymb}
\usepackage{amsmath}
\usepackage{amsthm}
\usepackage{verbatim}

\newtheorem{thm}{Theorem}[section]
\newtheorem{cor}[thm]{Corollary}
\newtheorem{lem}[thm]{Lemma}

\newtheorem{prop}[thm]{Proposition}
\theoremstyle{definition}
\newtheorem{Def}[thm]{Definition}
\newtheorem{eg}[thm]{Example}
\newtheorem*{eg*}{Example}

\newtheorem*{rem}{Remark}

\newcommand{\sgn}{\operatorname{sgn}}
\newcommand{\sym}[1]{\mathfrak{S}_{#1}}
\newcommand{\up}[1]{\text{$\uparrow$}^{#1}}
\newcommand{\down}[1]{\text{$\downarrow$}_{#1}}
\newcommand{\F}{\mathbb{F}}
\newcommand{\soc}{\operatorname{soc}}

\newcommand{\e}[1]{\varepsilon(#1)}
\newcommand{\bigboxtimessymbol}{\text{\Large $\boxtimes$}}
\newcommand{\bigboxtimes}{\operatorname*{\bigboxtimessymbol}}

\newcommand{\Ext}{\operatorname{Ext}}
\newcommand{\rad}{\operatorname{rad}}

\begin{document}
\title[Complexities of simple modules]{The complexities of some simple modules\\ of the symmetric groups}

\author{Kay Jin Lim}
\author{Kai Meng Tan}
\address{Department of Mathematics, National University of Singapore, Block S17, 10 Lower Kent Ridge Road, Singapore 119076.}

\email[K. J. Lim]{matlkj@nus.edu.sg}
\email[K. M. Tan]{tankm@nus.edu.sg}

\date{May 2012}
\thanks{Supported by Singapore Ministry of Education Academic Research Fund R-146-000-135-112.}

\begin{abstract}
We show that the simple modules of the Rouquier blocks of symmetric groups, in characteristic $p$ and having $p$-weight $w$ with $w < p$, have a common complexity $w$, and that when $p$ is odd, $D^{(p+1,1^{p-1})}$ has complexity $1$, while the other simple modules labelled by a partition having $p$-weight $2$ have complexity $2$.
\end{abstract}

\maketitle

\section{Introduction} \label{S:intro}

Representation theory of groups, first studied by Frobenius in the late nineteenth century, has now become a major research area.  It was originally used as a tool to understand the underlying structure of groups, and a major triumph for representation theorists came when Burnside proved, using representation theory, that every finite group whose order is divisible by at most two distinct primes is soluble.  The area has grown so much over the last century that it is now of interest to not only group theorists and algebraists, but also number theorists, algebraic geometers and algebraic topologists.

While the ordinary representation theory of finite groups has been well developed, very little is known in comparison in the modular case.  The main difficulty in the latter case is that the group algebra of a finite group $G$ is not only non-semi-simple as an algebra, but is also of infinite representation type whenever the Sylow $p$-subgroups of $G$ are non-cyclic (where $p$ is the characteristic of the ground field).  Invariants such as vertices and complexities --- which may be viewed as a measure of how much a module differs from a projective module --- have been invented to put the indecomposable modules into classes.  However, these invariants for a given module are very difficult to compute in general, and are known in very few cases.

Surprising little of these invariants have been computed even for the symmetric groups.  The complexities of its simple modules are almost completely unknown.  The only non-trivial ones which have been computed so far (as far as we know) are of the completely splittable simple modules \cite{HN} and the simple modules labelled by regularised partitions of hook partitions of $n$ with $p \nmid n$ \cite{Lim1}.  (Actually the first author in \cite{Lim1} computes the complexity of the Specht module $S^{(n-r,1^r)}$, which is simple when $p \nmid n$ \cite{Peel} and hence isomorphic to $D^{(n-r,1^r)^R}$.)

In this paper, we study the complexities of the simple modules of the symmetric groups in two important classes.  The first consists of the simple modules lying in the Rouquier blocks of Abelian defect $w$.  These blocks are well understood by the work of Chuang and the second author \cite{Rouquier}.  We show that the simple modules in these blocks have a common complexity $w$.  Our approach here is to employ the Morita equivalence between the canonical Rouquier block and the principal block of $\F(\sym{p} \wr \sym{w})$ proved by Chuang and Kessar \cite{CK}, and study the simple modules of the latter block.

The second class consists of the simple modules lying in defect two blocks.  These blocks have $p$-weight $2$, where $p$ is odd, and are also well understood by the works of several authors over the last two decades.  We show that the simple module $D^{(p+1,1^{p-1})}$ has complexity $1$, while all the other simple modules lying in these blocks have complexity $2$. Our approach here is to use the implicit result of Schroll and the second author in \cite{ST} that each such simple module either `induces semi-simply' to a Rouquier block through a sequence of $[2:k]$-pairs, or `restricts semi-simply' to the principal block $B_0(\sym{2p})$ of $\sym{2p}$ through a sequence of $[2:k]$-pairs.  As the Rouquier blocks have already been dealt with (in our first class), we are thus reduced to the study of the simple modules of $B_0(\sym{2p})$, for which we rely on the earlier work of Mikaelian \cite{Mikaelian}.

The paper is organised as follows: we give a summary of the background theory in the next section.  In section \ref{S:Rouq}, we study the Rouquier blocks, while in section \ref{S:wt2}, we study the defect two blocks.

\section{Preliminaries} \label{S:prelim}

In this section, we provide the background theory that we require.  General references are \cite{DB,GJ}.  Throughout we fix an algebraically closed field $\F$ of characteristic $p$, and $\F_p$ be the unique subfield of $\F$ of order $p$. All modules are left modules.

\subsection{Blocks}
Let $G$ be a finite group.  The group algebra $\F G$
can be uniquely decomposed into a direct sum of indecomposable two-sided ideals, which are called {\em blocks}.
Each block $B$ of $\F G$ is associated to a unique primitive central idempotent $e_B$ of $\F G$.  Let $M$ be a nonzero $\F G$-module $M$.  We say that $M$ {\em lies} in $B$ if and only if $e_B$ acts as identity on $M$.

The {\em principal} block of $\F G$ is the block in which the trivial $\F G$-module $\F$ lies.

\subsection{Cohomological varieties of modules}

Let $G$ be a finite group.  Denote by $V_G$ the affine variety defined by the maximum ideal spectrum of the cohomology ring $H^\cdot(G,\F) = \Ext^{\cdot}_{\F G} (\F, \F)$.

Let $M$ be a finite-dimensional $\F G$-module. The cohomological variety $V_G(M)$ of $M$ is defined to be the subvariety of $V_G$ consisting of maximal ideals of $H^\cdot(G,\F)$ containing the annihilator of $\Ext^*_{\F G}(M,M)$ (thus $V_G(\F) = V_G$). It is well known that $V_G(M)$ is homogeneous. We write $\overline{V_G(M)}$ for the projective variety associated to $V_G(M)$.

Let $H$ be a subgroup of $G$.  We write $V_H(M)$ for the cohomological variety of $M$ as a $\F H$-module.  This admits a natural action of $N_G(H)/C_G(H)$.  There is a morphism of affine varieties $\iota : V_H(M) \to V_G(M)$ such that $\iota(a) = \iota(b)$ if and only if $a$ and $b$ lie in the same $N_G(H)/C_G(H)$-orbit.

We collate together some results which we require.

\begin{thm}\label{T:cohomological variety}
Let $G$ be a finite group and let $M$ be a finite-dimensional $\F G$-module.
\begin{enumerate}
\item $V_G(M) = \bigcup_E \iota(V_E(M))$, where $E$ runs over representatives of conjugacy classes of maximal elementary Abelian $p$-subgroups of $G$.
\item If $M$ is indecomposable, then $\overline{V_G(M)}$ is connected in the Zariski topology.
\end{enumerate}
\end{thm}

\subsection{Rank varieties of modules}

Let $E$ be an elementary Abelian $p$-group isomorphic to $(C_p)^k$, with generators $g_1,g_2, \dotsc, g_k$.  For each $\alpha = (\alpha_1, \alpha_2,\dotsc, \alpha_k) \in \F^{k}$ with $\alpha \ne 0$, let $u_{\alpha} = 1+ \sum_{i=1}^k \alpha_i (g_i-1) \in \F E$.  Then $(u_{\alpha})^p = 1$.  Write $\langle u_{\alpha} \rangle$ for the cyclic group of order $p$ generated by $u_{\alpha}$.  Then the group algebra $\F\langle u_{\alpha} \rangle$ is subalgebra of $\F E$.

Let $M$ be a finite-dimensional $\F E$-module.  The {\em rank variety} $V^{\#}_E(M)$ of $M$ is defined as
\[
V^{\#}_E(M)=\{ \alpha \in \F^k \mid \alpha \ne 0,\  M \text{ is non-free as an $\F\langle u_{\alpha} \rangle$-module} \} \cup \{0\}.
\]
This is an affine subvariety of $\F^k$, and is independent of the choice and order of the generators (in the sense that two varieties obtained using different choices of generators are isomorphic).  More importantly, it is isomorphic to the cohomological variety $V_E(M)$ of $M$.  

\begin{thm} \label{T:coh}
Let $u$ and $v$ be units of $\F E$ of multiplicative order $p$ such that $u-1 \notin (\rad(\F E))^2$ and $u-v \in (\rad(\F E))^2$, and let $M$ be a finite-dimensional $\F E$-module.  Then $M$ is free as an $\F\langle u\rangle$-module if and only if $M$ is free as an $\F\langle v\rangle$-module.

In particular, if $\alpha = (\alpha_1,\dotsc, \alpha_k) \in (\F_p)^k$, then $M$ is free as an $\F\langle u_{\alpha}\rangle$-module if and only if $M$ is free as an $\F\langle \prod_{i=1}^k g_i^{\alpha_i}\rangle$-module.
\end{thm}

The above theorem also allows us to describe the natural action of $N_G(E)$ on $V^{\#}_E(M)$, when $E$ is an elementary Abelian $p$-subgroup of a finite group $G$:  if $n \in N_G(E)$ such that $ng_in^{-1} = \prod_{j=1}^k g_j^{a_{ij}}$ for each $i$, and $\alpha = (\alpha_1,\dotsc, \alpha_k) \in \F^k$, then $n \cdot \alpha = (\sum_{j=1}^k a_{1j}\alpha_j, \dotsc, \sum_{j=1}^k a_{kj}\alpha_j)$.



\subsection{Complexities}

Let $A$ be a finite-dimensional algebra over $\F$, and let $M$ be a finitely generated $A$-module.  Let
\[
\dotsb \to P_1 \to P_0 \to M
\]
be a minimal projective resolution of $M$.  The {\em complexity} of $M$, denoted $c_A(M)$, is defined as
\[
\min\{ n \in \mathbb{Z}_{\geq 0} \mid \lim_{r \to \infty} \tfrac{\dim_{\F}(P_r)}{r^n} = 0 \}.
\]

When $A = \F G$, where $G$ is a finite group, we write $c_G(M)$ for $c_A (M)$.  We collate together some well-known results about complexities of modules of finite group algebras.

\begin{thm} \label{T:complexity}
Let $G$ be a finite group, and let $M$ be a finitely generated $\F G$-module.
\begin{enumerate}
\item $c_G(M)$ equals the Krull dimension of $V_G(M)$.
\item $c_G(M) = 0$ if and only if $M$ is projective.
\item $c_G(M) = \max_E \{ c_E(M) \}$ where $E$ runs over representatives of conjugacy classes of maximal elementary Abelian $p$-subgroups of $G$.
\item If $G$ is a subgroup of $K$, then $c_K(M \up{K}) = c_G(M)$.
\item If $N$ is another finitely generated $\F G$-module, then $c_G(M \oplus N) = \max\{c_G(M), c_G(N)\}$.
\item If $K$ is another finite group, and $N$ is a finitely generated $\F K$-module, then
 $c_{G \times K}(M \boxtimes_\F N) = c_G(M) + c_K(N)$.
\end{enumerate}
\end{thm}

\subsection{Partitions}

A partition $\lambda
=(\lambda_{1},\lambda_{2},\dotsc)$ is a weakly decreasing sequence of non-negative integers, where for sufficiently large $i$, $\lambda_i =0$.  If $\sum_i \lambda_i = n$, we say $\lambda$ is a partition of $n$.  The length of $\lambda$, denoted $l(\lambda)$, equals $\max \{i \mid \lambda_i >0\}$.  The partition $\lambda$ is $p$-regular if there does not exist $i$ such that $\lambda_i = \lambda_{i+1} = \dotsb = \lambda_{i+p-1} > 0$.
Denote the set of partitions of
${n}$ by $\mathcal{P}_{n}$, and let
$\mathcal{P}=\bigcup_{n} \mathcal{P}_{n}$ be the set of all partitions.

A strictly decreasing sequence $\beta= (\beta_1,\beta_2,\dotsc, \beta_s)$ of non-negative integers is a sequence of $\beta$-numbers for $\lambda$ if $s \geq l(\lambda)$, and $\beta_i = \lambda_i + s - i$ for all $1 \leq i \leq s$.  Every strictly decreasing sequence of non-negative integers is a sequence of $\beta$-numbers for a unique partition.

The James $p$-abacus has $p$ vertical runners, labelled $0,1,\dotsc, p-1$ from left to right.  Its positions are labelled from left to right, and top down, starting from $0$.  The partition $\lambda$ may be displayed on the abacus as follows:  if $\beta =(\beta_1,\beta_2,\dotsc, \beta_s)$ is a sequence of $\beta$-numbers for $\lambda$, then we place a bead at position $\beta_i$ for each $i$.  This is the ($p$-)abacus display of $\lambda$ with $s$ beads.

In an abacus display of $\lambda$, moving a bead from position $a$ to a vacant position $b$, with $a>b$, corresponds to removing a (rim) hook of length $a-b$ from $\lambda$.  When we slide the beads as far up their respective runners as possible, we thus obtain an abacus display of the $p$-core of $\lambda$.  The $p$-weight of $\lambda$ is the total number of times we slide the beads one position up their respective runners to obtain its $p$-core.

A bead in an abacus display of $\lambda$ is {\em addable} (resp.\ {\em removable}) if its succeeding (resp.\ preceding) position is unoccupied.  Moving this bead to its succeeding (resp.\ preceding) unoccupied position corresponds to adding (resp.\ removing) a node to the Young diagram of $\lambda$ (and in so doing, obtain the Young diagram of another partition).  In addition, if this bead lie in the runner $i$ and the abacus display has $s$ beads, we call this bead $r$-addable (resp.\ $r$-removable) where $r$ is the residue class of $i-s+1$ (resp.\ $i-s$) modulo $p$.

An $r$-addable bead at position $t$ is {\em $r$-conormal} if the number of $r$-addable nodes strictly between positions $x$ and $t$ is always not less than that of the $r$-removable nodes strictly between positions $x$ and $t$ for all $x < t$. An $r$-removable bead at position $t$ is {\em $r$-normal} if the number of $r$-removable nodes strictly between positions $t$ and $y$ is always not less than that of the $r$-addable nodes strictly between positions $t$ and $y$ for all $y > t$.

\subsection{Symmetric groups}

Let $n \in \mathbb{Z}^+$.  Let $\sym{n}$ denote the permutation group on $\{1,2,\dotsc, n\}$.  For $m \in \mathbb{Z}^+$ with $m < n$, we embed $\sym{m}$ into $\sym{n}$ as the subgroup fixing $\{m+1,\dotsc, n\}$ pointwise.

To each $\lambda \in \mathcal{P}_n$, there is an associated Specht module $S^{\lambda}$ of $\sym{n}$, which has a simple head $D^{\lambda}$ if $\lambda$ is $p$-regular.  The set $\{ D^{\lambda} \mid \lambda \in \mathcal{P}_n, \lambda \text{ $p$-regular} \}$ is a complete set of pairwise non-isomorphic irreducible $\F \sym{n}$-modules.

Two Specht modules $S^{\lambda}$ and $S^{\mu}$ lie in the same block if and only if $\lambda$ and $\mu$ have the same $p$-core and the same $p$-weight.  We may thus speak of the $p$-core and the $p$-weight of a block $B$ of $\F\sym{n}$.  Write $\mathcal{P}_B = \{\lambda \in \mathcal{P}_n \mid S^\lambda \text{ lies in } B \}$.

\subsection{Restriction and induction} \label{SS:resind}

Let $B$ be a block of $\mathbb{F}\sym{n}$, with $p$-weight $w$ and $p$-core $\kappa_B$ (so $\kappa_B \in \mathcal{P}_{n-wp}$).  Fix an abacus display of $\kappa_B$ with $s$ beads, and consider the $p$-core partition $\kappa_C$ having an abacus display in which all the runners have the same number of beads as those in $\kappa_B$ except for runners $i-1$ and $i$, where respectively there are $k$ beads more and $k$ beads less than those in $\kappa_B$.  We assume that $\kappa_C \in \mathcal{P}_m$ with $m \leq n-k$, so that there is a block of $\mathbb{F}\sym{n-k}$ with $p$-core $\kappa_C$; we denote this block as $C$.

Let $r$ be the residue class of $(i-s)$ modulo $p$.  For each partition $\lambda \in \mathcal{P}_B$ having $k$ or more $r$-normal beads, let $\Phi(\lambda)$ be the partition obtained from $\lambda$ by moving the $k$ topmost $r$-normal beads to their respective unoccupied preceding positions.  For each partition $\mu \in \mathcal{P}_C$ having $k$ or more $r$-conormal beads, let $\Psi(\mu)$ be the partition obtained from $\mu$ by moving the $k$ bottommost $r$-conormal beads to their respective unoccupied succeeding positions. Then $\Phi = \Phi_{B,C}$ and $\Psi = \Psi_{B,C}$ are inverses of each other, inducing a bijection between the partitions in $B$ having $k$ or more $r$-normal beads and the partitions in $C$ having $k$ or more $r$-conormal beads.

\begin{thm}[{\cite[Section 11.2]{Kleshbook}}] \label{T:Klesh}
Let $\lambda \in \mathcal{P}_B$ and $\mu \in \mathcal{P}_C$.
\begin{enumerate}
\item If $\lambda$ has $k$ or more $r$-normal beads, then $\lambda$ is $p$-regular if and only if $\Phi(\lambda)$ is.
\item If $\mu$ has $k$ or more $r$-conormal beads, then $\mu$ is $p$-regular if and only if $\Psi(\mu)$ is.
\item
\[\soc(D^\lambda \down{C}) =
\begin{cases}
(D^{\Phi(\lambda)})^{\oplus k!}, & \text{if $\lambda$ has $k$ or more $r$-normal beads}; \\
0, &\text{otherwise}.
\end{cases}
\]
In addition, $D^\lambda \down{C} = (D^{\Phi(\lambda)})^{\oplus k!}$ if and only if $\lambda$ has exactly $k$ $r$-normal beads.

\item
\[\soc(D^\mu \up{B}) =
\begin{cases}
(D^{\Psi(\mu)})^{\oplus k!}, & \text{if $\mu$ has $k$ or more $r$-conormal beads}; \\
0, &\text{otherwise}.
\end{cases}
\]
In addition, $D^\mu \up{B} = (D^{\Psi(\mu)})^{\oplus k!}$ if and only if $\mu$ has exactly $k$ $r$-conormal beads.
\end{enumerate}
\end{thm}

Now suppose further that runner $i$ of the abacus display of $\kappa_B$ has exactly $k$ beads more than runner $i-1$.  Then $C$ also has $p$-weight $w$, and the abacus display of $\kappa_C$ can be obtained from that of $\kappa_B$ by interchanging the runners $i$ and $i-1$.  In this case, we call $(B,C)$ a $[w:k]$-pair.  Now, every partition in $\mathcal{P}_B$ has at least $k$ $r$-normal beads and every partition in $\mathcal{P}_C$ has at least $k$ $r$-conormal beads, so that $\Phi$ (as well as $\Psi$) is a bijection between the entire $\mathcal{P}_B$ and the entire $\mathcal{P}_C$.  In fact, most partitions in $\mathcal{P}_B$ has exactly $k$ $r$-normal beads, so that such $p$-regular partitions label the simple modules lying in $B$ that remain semi-simple upon restriction to $C$ --- we call these simple modules lying in $B$ {\em non-exceptional} (for the $[w:k]$-pair $(B,C)$), and the others {\em exceptional}.  Analogously, we define the exceptional and non-exceptional simple modules of $C$ (for $(B,C)$).  We note that the bijection $\Phi$ sends partitions labelling exceptional (resp.\ non-exceptional) simple modules of $B$ to partitions labelling exceptional (resp.\ non-exceptional) simple modules of $C$.

A necessary and sufficient condition for the absence of exceptional simple modules of $B$ and $C$ is $w \leq k$.  In this case, $B$ and $C$ are Morita equivalent \cite{Scopes} --- following Fayers \cite{F}, we shall say that $B$ and $C$ are Scopes equivalent --- and the effect of $\Phi$ on a partition in $B$ is merely to interchange runners $i$ and $i-1$ of its abacus display.  It is clear that Scopes equivalence can be extended to an equivalence relation on the set of all blocks of symmetric group algebras.

\begin{Def}
Let $\lambda$ be a $p$-regular partition with $p$-weight $w$.  We say that $D^{\lambda}$ {\em induces semi-simply} (resp.\ {\em restricts semi-simply}) to a block $B$ if and only if there exists a sequence $B_0,\dotsc, B_s$ of blocks such that
\begin{itemize}
\item $\lambda \in \mathcal{P}_{B_0}$;
\item $B_s = B$;
\item for each $1 \leq i \leq s$, $(B_i,B_{i-1})$ (resp.\ $(B_{i-1},B_i)$) is a $[w:k_i]$-pair for some $k_i \in \mathbb{Z}^+$;
\item for each $1 \leq j \leq s$, $(\Psi_{B_j,B_{j-1}} \circ \Psi_{B_{j-1},B_{j-2}} \circ \dotsb \circ \Psi_{B_1,B_0})(\lambda)$ (resp.\ $(\Phi_{B_{j-1},B_{j}} \circ \Phi_{B_{j-2},B_{j-1}} \circ \dotsb \circ \Phi_{B_0,B_1})(\lambda)$) is non-exceptional for $(B_j,B_{j-1})$ (resp.\ $(B_{j-1},B_j)$).
\end{itemize}
\end{Def}

\section{Rouquier blocks} \label{S:Rouq}

In this section, we show that every Rouquier block having $p$-weight $w$ with $w < p$ has the property that their simple modules have a common complexity $w$.

Let $B$ be a block of $\mathbb{F}\sym{n}$ with $p$-weight $w$.  We say that $B$ is a {\em Rouquier} block if the abacus display of its $p$-core has the following property:  for all $0 \leq i < j < p$, either runner $j$ has at least $w-1$ beads more than runner $i$, or runner $i$ has at least $w$ beads more than runner $j$.  It is easy to check that such a property is independent of the choice of abacus display of the $p$-core of $B$, and also not difficult to verify that the Rouquier blocks of a fixed weight form a single Scopes equivalence class.  These blocks were singled out by Rouquier in 1991 in his attempt to verify the truth of Brou\'e's Abelian defect group conjecture for symmetric groups.  He believed these blocks to have special properties, and in the case where $w < p$, he conjectured that they are Morita equivalent to the principal block of $\F(\sym{p} \wr \sym{w})$.  This conjecture was proved by Chuang and Kessar:

\begin{thm}[{\cite{CK}}] \label{T:Morita}
Let $B$ be a Rouquier block with $p$-weight $w$, where $w < p$.  Then $B$ is Morita equivalent to the principal block of $\F(\sym{p} \wr \sym{w})$.
\end{thm}

As the complexities of modules are preserved under Morita equivalence, to obtain the complexity of simple modules of the Rouquier blocks with $p$-weight $w$, where $w < p$, it suffices to study the simple modules lying in the principal block of $\F(\sym{p} \wr \sym{w})$.  The latter may be constructed as follows (see \cite[Proposition 3.7]{wreath}).

The simple $\F\sym{w}$-modules remain simple when inflated to $\sym{p} \wr \sym{w}$ and lie in the principal block of $\F(\sym{p} \wr \sym{w})$.  Also if $S$ is a simple $\F\sym{p}$-module lying in its principal block, then $S^{\otimes w}$ is a simple $\F(\sym{p} \wr \sym{w})$-module lying in its principal block, where the base group $(\sym{p})^w$ acts component-wise and the quotient $\sym{w}$ permutes the components.  In general, the simple modules of $\F(\sym{p} \wr \sym{w})$ lying in its principal block have the following form:
\[
\left(\bigboxtimes_{i=1}^k (S_i^{\otimes a_i} \otimes_\F T_i) \right) \up{\sym{p} \wr \sym{w}}_{\prod_{i=1}^k (\sym{p} \wr \sym{a_i})}
\]
where $a_1,\dotsc, a_k \in \mathbb{Z}^+$ such that $\sum_{i=1}^k a_i = w$, $S_1,\dotsc, S_k$ are pairwise non-isomorphic simple $\F\sym{p}$-modules lying in its principal block, and $T_1,\dotsc, T_k$ are (possibly isomorphic) simple $\F\sym{a_i}$-modules inflated to $\sym{p} \wr \sym{a_i}$.

\begin{thm} \label{T:wreath}
Let $w < p$.  The simple modules lying in the principal block of $\F(\sym{p} \wr \sym{w})$ have a common complexity $w$.
\end{thm}

\begin{proof}
We have, by Theorem \ref{T:complexity}(iii,iv,vi),
\begin{align*}
c_{\sym{p} \wr\sym{w}}\left( \left(\bigboxtimes_{i=1}^k (S_i^{\otimes a_i} \otimes_\F T_i) \right) \up{\sym{p} \wr \sym{w}}_{\prod_{i=1}^k (\sym{p} \wr \sym{a_i})} \right)
&= \sum_{i=1}^k c_{\sym{p} \wr \sym{a_i}} (S_i^{\otimes a_i} \otimes_\F T_i) \\
&= \sum_{i=1}^k c_{E_{a_i}} (S_i^{\otimes a_i} \otimes_\F T_i)
\end{align*}
where for each $i$, $E_{a_i}$ is a Sylow $p$-subgroup of $\sym{p} \wr \sym{a_i}$, which is elementary Abelian.  We may choose $E_{a_i}$ to be a subgroup of the base group $(\sym{p})^{a_i}$.  Then $E_{a_i}$ acts trivially on $T_i$, so that as $\F E_{a_i}$-modules $S_i^{\otimes a_i} \otimes_\F T_i$ is isomorphic to $\dim(T_i)$ copies of $S_i^{\otimes a_i}$.  Thus by Theorem \ref{T:complexity}(v,vi), we have
\[
c_{E_{a_i}}( S_i^{\otimes a_i} \otimes_\F T_i) = c_{E_{a_i}}(S_i^{\otimes a_i}) = c_{E_{a_i}}(S_i^{\boxtimes a_i}) = a_i \cdot c_{C_p}(S_i) = a_i
\]
since $c_{C_p}(S_i) = c_{\sym{p}}(S_i) = 1$, as is well known.
\end{proof}

\begin{thm} \label{T:Rouq}
Let $w < p$.  The simple modules lying in the Rouquier block with $p$-weight $w$ have a common complexity $w$.
\end{thm}

\begin{proof}
This follows immediately from Theorems \ref{T:Morita} and \ref{T:wreath}.
\end{proof}

The complexities of some other simple modules lying in other blocks with $p$-weight $w$, where $w <p$, can be obtained using Theorem \ref{T:Rouq} and the following proposition.

\begin{prop} \label{P:same}
Let $H$ be a subgroup of a finite group $G$.  Suppose that $M$ and $N$ are finitely generated $\F G$- and $\F H$-modules respectively, satisfying the following:
\begin{align*}
M \mid N \up{G}, \qquad
N \mid M \down{H}.
\end{align*}
Then $c_G(M) = c_H(N)$.
\end{prop}

\begin{proof}
By Theorem \ref{T:complexity}(iii,iv,v),
\[
c_H(N) = c_G(N \up{G}) \geq c_G(M) \geq c_H(M \down{H}) \geq c_H(N),\]
forcing equality throughout.
\end{proof}

\begin{cor} \label{C:same}
Let $(B,C)$ be a $[w:k]$-pair, and let $D^{\lambda}$ be a non-exceptional simple module lying in $B$.  Then $D^{\lambda}$ and $D^{\Phi(\lambda)}$ have the same complexity.
\end{cor}

\begin{proof}
Note that $D^{\lambda}$ and $D^{\Phi(\lambda)}$ satisfy the hypothesis of $M$ and $N$ in Proposition \ref{P:same}.
\end{proof}

\begin{cor} \label{C:inducetoRouq}
Let $\lambda$ be a $p$-regular partition of $n$ with $p$-weight $w$, where $w < p$.  If $D^{\lambda}$ induces semi-simply to a Rouquier block, then $c_{\sym{n}}(D^{\lambda}) = w$.
\end{cor}

\begin{proof}
This follows immediately from Theorem \ref{T:Rouq} and Corollary \ref{C:same}.
\end{proof}

We also note the existence of simple modules labelled by partitions with $p$-weight $w$ ($w \geq 2$) having complexity less than $w$, where possibly $w \geq p$.

\begin{prop} \label{P:comp<}
Let $w \in \mathbb{Z}$, $w \geq 2$, and assume that $p$ is odd.  Then $$c_{\sym{wp}}(D^{((w-1)p+1,1^{p-1})}) < w.$$
\end{prop}

\begin{proof}
Let $\lambda = ((w-1)p+1,1^{p-1})$.
Note first that $D^{\lambda} \down{\sym{wp-1}} = D^{((w-1)p, 1^{p-1})}$ by Theorem \ref{T:Klesh}(iii).  Since $((w-1)p, 1^{p-1})$ has $p$-weight $w-2$, we thus have (see \cite[Section 2, last sentence]{Lim2})
$$c_{\sym{wp-1}}(D^{\lambda}) \leq w-2.$$

Let $E$ be a maximal elementary Abelian $p$-subgroup of $\sym{wp}$.  Then $E \cong (C_p)^k$ with $k \leq w$, so that $V^{\#}_E(D^{\lambda})$ is an affine subvariety of $\F^k$, and $c_E(D^{\lambda}) \leq k$.  If $c_E(D^{\lambda}) = w$, then $k = w$ and $V^{\#}_E(D^{\lambda}) = \F^w$, so that $V^{\#}_{E'}(D^{\lambda}) = \F^{w-1}$ for any subgroup $E'$ of $E$ of order $p^{w-1}$.  Furthermore, since $p$ is odd, $E$ is generated by $w$ disjoint $p$-cycles.  As such, we may choose $E'$ to be a subgroup of both $E$ and $\sym{wp-1}$, and obtain the following contradiction:
$$
w-1 = c_{E'}(D^{\lambda}) \leq c_{\sym{wp-1}}(D^{\lambda}) \leq w-2.
$$
Thus $c_{\sym{wp}} (D^{\lambda}) = \max_E(c_E(D^{\lambda})) \leq w-1$.
\end{proof}

\begin{rem} Proposition \ref{P:comp<} does not hold in general when $p=2$. Let $E=\langle (12)(34),(13)(24)\rangle\subseteq \sym{4}$. Then $D^{(3,1)}\down{E}\cong \F\oplus \F$ and hence $c_{\sym{4}}(D^{(3,1)})=2$.
\end{rem}

\section{Defect two blocks} \label{S:wt2}

In this section, we study the defect two blocks of symmetric group algebras.  These blocks have $p$-weight $2$, where $p$ is odd, and are fairly well understood by the work of many authors, though the complexities of their simple modules have hitherto not been computed.  For our purposes, \cite{ST} and \cite{Mikaelian} are the most useful references.

There are several different labellings for partitions having $p$-weight $2$.  The one that suits our purpose most is that used in \cite{ST}, as follows.

\begin{Def}
Let $\lambda$ be a partition having $p$-weight $2$.
\begin{itemize}
\item If the abacus display of $\lambda$ has a bead at position $x$ and a bead at position $x-p$, while position $x-2p$ is vacant, and there are exactly $a$ vacant positions strictly between $x$ and $x-p$, and $b$ vacant positions strictly between $x-p$ and $x-2p$, then
    $$\lambda = [a,b] \quad \text{and} \quad \e{\lambda} = 0.$$

\item If the abacus display of $\lambda$ has a bead at position $x$ which two vacant positions above it, i.e.\ positions $x-p$ and $x-2p$ are vacant, and there are exactly $a$ vacant positions strictly between $x$ and $x-p$, and $b$ vacant positions between $x-p$ and $x-2p$, inclusive of $x-p$ but not $x-2p$, then
    $$\lambda = [a,b] \quad \text{and} \quad \e{\lambda} = 1.$$

\item If the abacus display of $\lambda$ has two beads, at positions $x$ and $y$ say, with $x > y$, $x \not\equiv y \pmod p$, each with a vacant position above it, i.e.\ positions $x-p$ and $y-p$ are vacant, and there are exactly $a$ vacant positions strictly between $x$ and $x-p$, and $b$ vacant positions strictly between $y$ and $y-p$, then
    $$\lambda = [a,b] \quad \text{and} \quad
\e\lambda =
\begin{cases}
1, &\text{if }x-p < y < x; \\
0, &\text{otherwise.}
\end{cases}
$$
\end{itemize}
\end{Def}

When we need to emphasize that $[a,b]$ is a partition in the weight $2$ block $B$, we write it as $[a,b]_B$.  Clearly, this labelling of $\lambda$ is independent of the abacus used to display $\lambda$.

\begin{eg} \hfill \label{E:eg}
\begin{enumerate}
\item Let $B$ be a Rouquier block having $p$-weight $2$, and fix an abacus display of its $p$-core which has at least two beads on each runner.  No two runners on this abacus have the same number of beads.  Let $(j_0,j_1,\dotsc, j_{p-1})$ be a permutation of $(0,1,\dotsc, p-1)$ such that runner $j_i$ has more beads than runner $j_{i-1}$ for all $1 \leq i <p$.  The partition obtained from its $p$-core by sliding the bottom bead on runner $j_a$ down two positions is labelled by $[a,a+1]_B$.  The partition obtained by sliding two distinct beads on runner $j_a$ and $j_b$ (with $a \geq b$) down one position each is labelled by $[a,b]_B$.  The $p$-singular partitions are $[0,1]_B$ and $[0,0]_B, [1,0]_B,\dotsc, [p-1,0]_B$.  Also,
\[
\e{[a,b]_{B}} = \begin{cases}
1, &\text{if } b = a+1;\\
0, &\text{otherwise.}
\end{cases}
\]
\item Let $B_0$ be the principal block of $\F\sym{2p}$.  We display the partitions in $B_0$ on an abacus having two beads on each runner.  The partition obtained from its $p$-core by sliding the bottom bead on runner $a$ down two positions is labelled by $[p-1,a+1]_{B_0}$.  The partition obtained by sliding one bead on runner $a$ and one bead on runner $b$ ($a > b$) down one position each is labelled by $[a-1,b+1]_{B_0}$.  The partition obtained by sliding two beads on runner $a$ down one position each is labelled by $[a,0]_{B_0}$.  Also,
\[
\e{[a,b]_{B_0}} = \begin{cases}
0, &\text{if } b = 0;\\
1, &\text{otherwise.}
\end{cases}
\]
\end{enumerate}
\end{eg}

This labelling is first introduced as $\lambda_B(a,b)$ in \cite[Proposition 98]{CTu}.  While it has the disadvantage that the determination of the actual partition from a given label $[a,b]_B$ may not be straightforward, its importance lies in the fact that over a $[2:k]$-pair, it is invariant under the action of $\Phi$.  More precisely,

\begin{lem}[{\cite[Lemmas 5.9 and 5.13]{ST}}] \label{L:invar}
Suppose that $(B,C)$ is a $[2:k]$-pair.  Let $\lambda \in \mathcal{P}_B$, and write $\Phi$ for $\Phi_{B,C}$.
\begin{enumerate}
\item If $\lambda = [a,b]_B$, then $\Phi(\lambda) = [a,b]_C$.
\item $\e\lambda \ne \e{\Phi(\lambda)}$ if and only if $k=1$ and ($\lambda$ is $p$-regular and) $D^\lambda$ is the unique simple module lying in $B$ that is exceptional for $(B,C)$, in which case $\e\lambda = 0$ and $\e{\Phi(\lambda)} = 1$.
\end{enumerate}
\end{lem}

In particular, since every block can be restricted through a sequence of $[2:k]$-pairs to the principal block of $\F\sym{2p}$, Lemma \ref{L:invar}(i) and Example \ref{E:eg}(ii) show in particular that each $[a,b]_B$ labels a unique partition in $\mathcal{P}_B$.  Lemma \ref{L:invar} also gives the following immediate corollary which, though not explicitly stated in \cite{ST}, is well known among experts.

\begin{cor} \label{C:invar}
Let $\lambda$ be a $p$-regular partition with $p$-weight $2$.
\begin{enumerate}
\item If $\e\lambda =0$, then $D^{\lambda}$ induces semi-simply to a Rouquier block.
\item If $\e\lambda = 1$, then $D^{\lambda}$ restricts semi-simply to the principal block of $\F\sym{2p}$.
\item If $\lambda = [a,a+1]$, then $\e\lambda =1$, and $D^{\lambda}$ also induces semi-simply to a Rouquier block.
\end{enumerate}
In particular, $D^{\lambda}$ either induces semi-simply to a Rouquier block or restricts semi-simply to the principal block of $\F\sym{2p}$.
\end{cor}

\begin{proof}
Suppose that $\lambda \in \mathcal{P}_B$.  Let $C_0,\dotsc, C_r$ be a sequence of defect two blocks such that $C_0 = B$, $C_r$ is a Rouquier block, and for each $1 \leq i \leq r$, $(C_i,C_{i-1})$ is a $[2:k_i]$-pair for some $k_i \in \mathbb{Z}^+$ (for the existence of this sequence, see \cite[Lemma 3.1]{F}).  Also, let $B_0,\dotsc, B_s$ be a sequence of defect two blocks such that $B_0$ is the principal block of $\F\sym{2p}$, $B_s = B$, and for each $1 \leq j \leq s$, $(B_j,B_{j-1})$ is a $[2:l_j]$-pair for some $l_j \in \mathbb{Z}^+$.

If $\e\lambda = 0$, then $D^{\lambda}$ induces semi-simply to $C_r$ through the sequence $C_0,\dotsc, C_r$ by Lemma \ref{L:invar}(ii).  Similarly, if  $\e\lambda = 1$, then $D^{\lambda}$ restricts semi-simply to $B_0$ through the sequence $B_s,\dotsc,B_0$ by Lemma \ref{L:invar}(ii).

If $\lambda = [a,a+1]_B$, then $(\Psi_{C_r,C_{r-1}} \circ \dotsb \circ \Psi_{C_1,C_0})(\lambda) = [a,a+1]_{C_r}$ by Lemma \ref{L:invar}(i), and $\e{[a,a+1]_{C_r}} = 1$ by Example \ref{E:eg}(i).  Thus by Lemma \ref{L:invar}(ii), $\e\lambda = 1$ and $D^{\lambda}$ induces semi-simply to $C_r$ through the sequence $C_0,\dotsc, C_r$.
\end{proof}

\begin{cor} \label{C:wt2}
Let $\lambda$ be a $p$-regular partition having $p$-weight $2$.  Then $D^{\lambda}$ has complexity $2$ if $\e\lambda = 0$ or $\lambda = [a,a+1]$.
\end{cor}

\begin{proof}
This follows immediately from Corollaries \ref{C:invar}(i,iii) and \ref{C:inducetoRouq}.
\end{proof}

To obtain the complexity of $D^{\lambda}$ for $p$-regular $\lambda$ having $p$-weight $2$ which is not covered by Corollary \ref{C:wt2}, we look at the principal block of $\F\sym{2p}$, which from now on will be denoted as $B_0$.

We start with the following well-known lemma.

\begin{lem} \label{L:restrictS_p}
Let $\lambda \in \mathcal{P}_{B_0}$ be $p$-regular, and $\lambda \ne (p+1,1^{p-1})\ (= [p-1,1]_{B_0})$.  Then $D^{\lambda} \down{\sym{2p-1}}$ has a simple summand labelled by a partition with $p$-weight $1$.

In particular, $D^{\lambda} \down{\sym{p}}$ is not projective.
\end{lem}

\begin{proof}
This follows from Theorem \ref{T:Klesh}(iii).
\end{proof}

Let $C_p$ denote the subgroup of $\sym{p}$ generated by the distinguished $p$-cycle $(1,2,\dotsc,p)$.  Also, write $\Delta C_p$ for the subgroup of $\sym{2p}$ generated  by $(1,2,\dotsc, p)(p+1,p+2,\dotsc, 2p)$. Our main result of this section relies heavily on the following proposition.

\begin{prop} \label{P:messy}
Let $\lambda \in \mathcal{P}_{B_0}$ be $p$-regular, $\lambda \ne (p+1,1^{p-1})$, and $\lambda \ne [a,a+1]_{B_0}$ for all $a$.  Then $D^{\lambda} \down{\Delta C_p}$ is not free.
\end{prop}

The proof of Proposition \ref{P:messy} involves some messy computations building on the work of Mikaelian \cite{Mikaelian}, and shall be deferred to the end of this section.

\begin{cor} \label{C:B_0}
Let $\lambda \in \mathcal{P}_{B_0}$ be $p$-regular.  Then
\[
c_{\sym{2p}}(D^{\lambda}) =
\begin{cases}
1, &\text{if } \lambda = (p+1,1^{p-1}); \\
2, &\text{otherwise}.
\end{cases}
\]
\end{cor}

\begin{proof}
Firstly, $c_{\sym{2p}}(D^{(p+1,1^{p-1})}) <2$ by Proposition \ref{P:comp<}.  Since $D^{(p+1,1^{p-1})}$ is not projective, by Theorem \ref{T:complexity}(ii), we must have $c_{\sym{2p}}(D^{(p+1,1^{p-1})}) = 1$.

If $\lambda = [a,a+1]$ for some $1 \leq a < p$, then Corollary \ref{C:wt2} applies.

For the remaining cases, both $D^{\lambda} \down{C_p}$ and $D^{\lambda} \down{\Delta C_p}$ are not free by Lemma \ref{L:restrictS_p} and Proposition \ref{P:messy}, and hence $(1,0),(1,1) \in V^{\#}_E(D^{\lambda})$ by Theorem \ref{T:coh}, where $E$ is the Sylow $p$-subgroup of $\sym{2p}$ generated by $g_1 = (1,2,\dotsc, p)$ and $g_2 = (p+1,p+2,\dotsc, 2p)$.  Since $g_1$ and $g_1 g_2$ are not conjugates in $\sym{2p}$, we see that $V_{\sym{2p}}(D^{\lambda})$ contains at least two distinct lines.  Consequently, this variety must have Krull dimension at least $2$ by Theorem \ref{T:cohomological variety}(ii), and thus equal to $2$.
\end{proof}

\begin{thm} \label{T:wt2}
Let $\lambda$ be a $p$-regular partition of $n$ with $p$-weight $2$.  Then
\[
c_{\sym{n}}(D^{\lambda}) =
\begin{cases}
1, &\text{if } \lambda = (p+1,1^{p-1}); \\
2, &\text{otherwise.}
\end{cases}
\]
\end{thm}

\begin{proof}
By Corollary \ref{C:wt2}, it suffices to consider the case where $\e\lambda = 1$ and $\lambda \ne [a,a+1]$ for all $a$.  In this case, $\lambda$ restricts semi-simply to $B_0$ by \ref{C:invar}(ii), so that, unless $\lambda = [p-1,1]$, we have $c_{\sym{n}}(D^{\lambda}) = 2$ by Lemma \ref{L:invar}(i) and Corollary \ref{C:same}.

It remains to show that if $\e\lambda = 1$ and $\lambda = [p-1,1]$, then $\lambda = (p+1,1^{p-1})$.  Let $B_s, \dotsc, B_1, B_0$ be the blocks through which $D^{\lambda}$ restricts semi-simply to $B_0$.  Since $(B_1,B_0)$ is a $[2:k]$-pair for some $k$, we see that $B_1$ must be the principal block of $\F\sym{2p+1}$, as this is the only block forming a $[2:k]$-pair for some $k$ with $B_0$.  Now, $D^{(p+1,1^{p-1})}$ is the unique exceptional simple module for $(B_1,B_0)$.  Thus, unless $s=0$, we would have $\e\lambda= 0$ by Lemma \ref{L:invar}(ii), a contradiction.  Hence $s=0$, and $\lambda = [p-1,1]_{B_0} = (p+1,1^{p-1})$.
\end{proof}

The remainder of this section shall be devoted to the proof of Proposition \ref{P:messy}.

For $1 \leq i < p$, let $D_i = D^{(i+1,1^{p-i-1})}$.  These are the simple modules lying in the principal block of $\F\sym{p}$.  Also, $D_i \down{C_p}$ has a unique non-projective summand, which is the Jordan block $J_i$ of size $i$ if $i$ is odd, and the Jordan block $J_{p-i}$ of size $p-i$ if $i$ is even (see \cite[\S 3]{Tsushima}).

Denote the principal block of $\F(\sym{p} \times \sym{p})$ by $b_0$.  The simple modules lying in $b_0$ are $D_i \boxtimes_\F D_j$ ($1 \leq i, j < p$).  For $1 \leq j \leq i < p$, let
\[\nabla(i,j) = \bigoplus_{k= j}^{i} D_k \boxtimes_\F D_{i+j-k}.\]

Identifying $\sym{p} \times \sym{p}$ as the subgroup $\prod_{i=1}^2 \sym{p}[i]$ of $\sym{2p}$, where $\sym{p}[i] = \{ \sigma[i] \mid \sigma \in \sym{p}\}$ and
$$(\sigma[i]) (j) =
\begin{cases}
\sigma(j-(i-1)p) + (i-1)p, &\text{if } (i-1)p < j \leq ip, \\
j, &\text{otherwise;}
\end{cases}
$$
we have the following theorem:

\begin{thm}[{\cite[Theorem 10]{Mikaelian}}] \label{T:Mik}
Let $a,b \in \mathbb{Z}$ such that $2 \leq b \leq a \leq p-2$ and $a+b \leq p$.  Then $D^{[a,b]_{B_0}} \down{b_0}$ is stable, and has Loewy length $3$, with head and socle both isomorphic to $\nabla(a,b)$ and a semi-simple heart isomorphic to $\nabla(a+1,b) \bigoplus \nabla(a,b-1)$.
\end{thm}

\begin{rem}
To obtain Theorem \ref{T:Mik} from Theorem 10 of \cite{Mikaelian}, please note the following relations between our notations and those in \cite{Mikaelian}:
\begin{align*}
D_i \boxtimes_{\F} D_j &= (p-i-1,p-j-1); \\
[a,b] &= (a-b+1|2p-a-b-1).
\end{align*}
\end{rem}

For any finite dimensional $\F G$-module $M$, we write $\Omega^0(M)$ for the projective-free part of $M$.  Thus, $\Omega^0(M)$ has no non-zero projective direct summand, and $M \cong \Omega^0(M) \oplus P$ for some projective $\F G$-module $P$.  Clearly, $\Omega^0(M)$ is well-defined up to isomorphism by Krull-Schmidt Theorem.

\begin{prop} \label{P:Deltarestrict}
For $1 \leq j \leq i < p$,
\[
\Omega^0(\nabla(i,j) \down{\Delta C_p}) \cong
\begin{cases}
\bigoplus_{k=1}^{\min\{\frac{i+j}{2},p-\frac{i+j}{2}\}} \min\{2k-1,i-j+1\}J_{2k-1}, &\text{if } 2 \mid (i+j); \\
\bigoplus_{k=1}^{\min\{\frac{i+j-1}{2},p-\frac{i+j+1}{2}\}} \min\{2k,i-j+1\} J_{p-2k}, &\text{otherwise.}
\end{cases}\]
\end{prop}

\begin{proof}
This Proposition uses the following identity (see Theorem 2.7(i,iv) in Chapter VIII of \cite{Feit}): for $1\leq x,y \leq p-1$, we have $ J_x \otimes_\F J_y \cong \bigoplus_{k=1}^{\min\{x,y\}} J_{|x-y|+2k-1}$ if $x+y \leq p$, while $\Omega^0(J_x \otimes_\F J_y) \cong J_{p-x} \otimes_\F J_{p-y}$ if $x+y > p$.  We provide an outline for the case $2 \mid (i+j) \leq p$; the other cases are similar.  In this case $k \equiv i+j-k \pmod 2$ for all $k$. Notice that $(D_k\boxtimes_\F D_{i+j-k})\down{\Delta C_p}\cong (D_k\down{C_p})\otimes_\F (D_{i+j-k}\down{C_p})$, so that
\begin{align*}
\Omega^0(\nabla(i,j) \down{\Delta C_p}) &= \bigoplus_{k=j}^{i} \Omega^0 ((D_{k} \down{C_p}) \otimes_\F (D_{i+j-k} \down{C_p})) \\
&= \bigoplus_{k=j}^{i} \Omega^0 (\Omega^0(D_{k} \down{C_p}) \otimes_\F \Omega^0(D_{i+j-k} \down{C_p})) \\
&\cong \bigoplus_{\substack{j \leq k \leq i\\ k \text{ odd}}} \Omega^0(J_k \otimes_\F J_{i+j-k}) \oplus \bigoplus_{\substack{j \leq k \leq i\\ k \text{ even}}} \Omega^0(J_{p-k} \otimes_\F J_{p-i-j+k}) \\
&\cong \bigoplus_{\substack{j \leq k \leq i\\ k \text{ odd}}} \Omega^0(J_k \otimes_\F J_{i+j-k}) \oplus \bigoplus_{\substack{j \leq k \leq i\\ k \text{ even}}} \Omega^0(J_{k} \otimes_\F J_{i+j-k}) \\
&\cong \bigoplus_{k=j}^i \bigoplus_{l=1}^{\min\{k,i+j-k\}} J_{|i+j-2k|+2l-1} \\
&\cong \bigoplus_{k=1}^{\frac{i+j}{2}} \min\{2k-1,i-j+1\}J_{2k-1}.
\end{align*}
\end{proof}

We record two useful lemmas whose proofs are obtained by brute force computations using Proposition \ref{P:Deltarestrict}.

\begin{lem} \label{L:lem1}
Let $a,b \in \mathbb{Z}$ such that $2 \leq b \leq a \leq p-2$ and $a+b \leq p$.  Then $$\Omega^0(\nabla(a+1,b) \down{\Delta C_p}) \oplus \Omega^0(\nabla(a,b-1) \down{\Delta C_p})$$ has exactly $N$ indecomposable summands, where
$$
N =
\begin{cases}
\tfrac{1}{2}(a+3b-2)(a-b+2), &\text{if } 2 \mid (a+b); \\
\tfrac{1}{2}(a+3b-1)(a-b+1)+2b-1, &\text{if } 2 \nmid (a+b) \text{ and } a+b < p; \\
\tfrac{1}{2}(a+3b-1)(a-b+1)+(2b-1)-(a-b+2), &\text{if } 2 \nmid (a+b) \text{ and } a+b = p.
\end{cases}
$$
\end{lem}

\begin{lem} \label{L:lem2}
Let $a,b \in \mathbb{Z}$ such that $2 \leq b \leq a \leq p-2$ and $a+b \leq p$.  Let $Q$ be a $\F(\Delta C_p)$-module having a filtration with factors
\[
\Omega^0(\nabla(a,b)\down{\Delta{C_p}}),\ \Omega^0(\nabla(a+1,b)\down{\Delta{C_p}}) \oplus \Omega^0(\nabla(a,b-1)\down{\Delta{C_p}}),\ \Omega^0(\nabla(a,b)\down{\Delta{C_p}}).
\]
Then the dimension of $Q$ is
$$
\begin{cases}
\tfrac{p}{2}(a+3b-2)(a-b+2) + a(1-2b)-b, &\text{if } 2 \mid (a+b); \\
\tfrac{p}{2}(a+3b-1)(a-b+1)+ a(2b-1) + b, &\text{if } 2 \nmid (a+b) \text{ and } a+b < p; \\
\tfrac{p}{2}(a+3b-1)(a-b+1)+ a(2b-1) + b -p(a-b+2), &\text{if } 2 \nmid (a+b) \text{ and } a+b = p.
\end{cases}
$$
\end{lem}

We are now ready to prove Proposition \ref{P:messy}.

\begin{proof}[Proof of Proposition \ref{P:messy}]
We note first the following identity, which is a special case of Lemma 5.10 in \cite{ST}:
$$ D^{[a,b]_{B_0}} \otimes_{\F} \sgn \cong D^{[p-a,p-b]_{B_0}} \qquad (a\geq b).$$
Here $\sgn$ denotes the signature representation of $\F\sym{2p}$.

As $[p-1,b]_{B_0} = (p+b,1^{p-b})$, $D^{[p-1,b]}$ has dimension $\binom{2p-2}{p-b}$, which is coprime to $p$ if $2 \leq b < p$.  Thus $D^{[p-1,b]}\down{\Delta C_p}$ is not free.  Since $D^{[p-1,b]} \otimes_\F \sgn \cong D^{[p-b,1]}$, the same dimension argument shows that $D^{[p-b,1]} \down{\Delta C_p}$ is not free too.  (The astute reader will of course notice that this already proves $V^{\#}_E(D^{[p-1,b]}) = V^{\#}_E(D^{[p-b,1]}) = \F^2$ for all $2 \leq b < p$, where $E$ is any Sylow $p$-subgroup of $\sym{2p}$.)

We are thus left with $[a,b]$ with $2 \leq b \leq a \leq p-2$.  We consider first the case $a+b \leq p$.  By Theorem \ref{T:Mik} and Proposition \ref{P:Deltarestrict}, we see that $D^{[a,b]} \down{\Delta C_p}$ has a direct summand $Q$ having a filtration with factors
\[
\Omega^0(\nabla(a,b)\down{\Delta{C_p}}),\ \Omega^0(\nabla(a+1,b)\down{\Delta{C_p}}) \oplus \Omega^0(\nabla(a,b-1)\down{\Delta{C_p}}),\ \Omega^0(\nabla(a,b)\down{\Delta{C_p}}).
\]
If $D^{[a,b]} \down{\Delta C_p}$ is free, then $Q$ is free, isomorphic to a direct sum of $\dim(Q)/p$ copies of the uniserial indecomposable free $\F(\Delta C_p)$-module of dimension $p$, and thus the composition length of every semi-simple subquotient of $Q$ is bounded above by $\dim(Q)/p$.  However, by Lemmas \ref{L:lem1} and \ref{L:lem2}, and noting that $a(2b-1)+b < (p-1)(2b-1) + b = p(2b-1) -b+1 < p(2b-1)$, we see that the number $N$ of indecomposable summands in $\Omega^0(\nabla(a+1,b)\down{\Delta{C_p}}) \oplus \Omega^0(\nabla(a,b-1)\down{\Delta{C_p}})$ --- a subquotient of $Q$ --- is greater than $\dim(Q)/p$, so that $Q$, and hence $D^{[a,b]} \down{\Delta C_p}$, cannot be free.

Finally, if $a+b > p$, then $D^{[a,b]} \otimes_{\F} \sgn \cong D^{[p-a,p-b]}$, so that
$$
D^{[a,b]} \down{\Delta C_p} \cong (D^{[a,b]} \otimes_{\F} \sgn) \down{\Delta C_p} \cong D^{[p-a,p-b]} \down{\Delta C_p}
$$
since the elements in $\Delta C_p$ are all even permutations in $\sym{2p}$.  The above paragraph applies and shows that $D^{[p-a,p-b]} \down{\Delta C_p}$ is not free.  Hence $D^{[a,b]} \down{\Delta C_p}$ is not free, and the proof is complete.
\end{proof}

\end{document}